\documentclass[11pt]{article}
\usepackage{amsmath}
\usepackage{amsthm}
\usepackage{amssymb}
\usepackage{enumerate}
\usepackage{xcolor}
\usepackage{cases}
\newtheorem{theo}{Theorem}[section]

\newtheorem{prop}[theo]{Proposition}
\newtheorem{lemma}[theo]{Lemma}

\newtheorem{conj}[theo]{Conjecture}

\newtheorem{claim}[theo]{Claim}
\newcommand{\dd}{\delta}
\newcommand{\sstar}{{s^*}}

\newcommand{\eps}{{\varepsilon}}
\begin{document}
\date{}

\title{
Irregular subgraphs
}

\author{Noga Alon
\thanks
{Department of Mathematics, Princeton University,
Princeton, NJ 08544, and
Schools of Mathematics and
Computer Science, Tel Aviv University, Tel Aviv 69978,
Israel.  
Email: {\tt nalon@math.princeton.edu}.  
Research supported in part by
NSF grant DMS-1855464, BSF grant 2018267
and the Simons Foundation.}
\and 
Fan Wei
\thanks
{
Department of Mathematics, Princeton University, Princeton, NJ
08544. 
Email: {\tt fanw@princeton.edu}.
Research supported by NSF Award
DMS-1953958.} 
}

\maketitle
\begin{abstract}
We suggest two related conjectures dealing with the existence of
spanning irregular subgraphs of graphs. The first asserts that any
$d$-regular graph on $n$  vertices contains a spanning subgraph in
which the number of vertices of each degree between $0$ and $d$
deviates from $\frac{n}{d+1}$ by at most $2$. The second is that every
graph on $n$ vertices with minimum degree $\delta$ contains a
spanning subgraph in which the number of vertices of each degree
does not exceed $\frac{n}{\delta+1}+2$. 
Both conjectures remain open, but we
prove several asymptotic relaxations for graphs with a large number
of vertices $n$. In particular we show that if
$d^3 \log n \leq o(n)$ then every $d$-regular graph with $n$
vertices contains a spanning subgraph in which the number of
vertices of each degree between $0$ and $d$ is 
$(1+o(1))\frac{n}{d+1}$.
We also prove that 
any graph
with $n$ vertices and minimum degree $\delta$ contains a spanning
subgraph in which no degree is repeated more than
$(1+o(1))\frac{n}{\delta+1}+2$ times.
\vspace{0.2cm}

subgraph, repeated degrees.
\end{abstract}

\section{Introduction}
All graphs considered here are simple, that is, contain no loops
and no parallel edges. For a graph $G$ and a nonnegative integer
$k$, let $m(G,k)$ denote the number of vertices of degree $k$ in
$G$, and let $m(G)=\max_k m(G,k)$ denote the maximum number of
vertices of the same degree in $G$. 
One of the basic facts in
Graph Theory is the statement that for every graph $G$ with at least $2$
vertices, $m(G) \geq 2$. 
In this paper we 
suggest the following two related conjectures.
\begin{conj}
\label{c11}
Every $d$-regular graph $G$ on $n$ vertices contains a spanning
subgraph $H$ so that for every $k$, $0 \leq k \leq d$,
$\left|m(H,k)-\frac{n}{d+1}\right|\leq 2.$
\end{conj}
\begin{conj}
\label{c12}
Every graph $G$ with $n$ vertices and minimum degree $\delta$
contains a spanning subgraph $H$ satisfying
$m(H) \leq \frac{n}{\delta+1} +2$.
\end{conj}
If true, both conjectures are tight. One example showing it is the
vertex disjoint union of two cycles of length $4$. There are also
many
examples showing that an extra 
additive $1$ is needed even 
when $\frac{n}{d+1}$ is an integer. Indeed,  if $G$ is any
$d$-regular graph with $n$ vertices, then by the pigeonhole
principle, for any spanning subgraph $H$ of $G$,
$m(H) \geq \frac{n}{d+1}$, as the degree of each vertex of $H$ is
an integer between $0$ and $d$. If, in addition, $n$ is
divisible by $d+1$, then the equality
$m(H)=\frac{n}{d+1}$ is possible only if $m(H,k)=\frac{n}{d+1}$
for any $0 \leq k \leq d$. However, this is impossible if
$\left(\lfloor \frac{d+1}{2} \rfloor \right) \frac{n}{d+1}$ 
is odd, as  the number
of vertices of odd degree in $H$ must be even. Note that a small
value of $m(H)$ can be viewed as a measure of the irregularity of
the graph $H$. Thus both conjectures address the question of the 
existence of highly irregular subgraphs of graphs, stating that
with this interpretation every graph $G$ contains a spanning subgraph
$H$ which is nearly as irregular as the degrees of $G$ permit.

We have not been able to prove any of the two conjectures above,
but can establish the following results, showing that some natural
asymptotic versions of both do hold. In the following two
results the $o(1)$ terms tend to $0$ as $n$ tends to infinity.
\begin{theo}
\label{t13}
If $d^3 \log n =o(n)$ then any $d$-regular graph with $n$ vertices
contains a spanning subgraph $H$ so that for every $0 \leq k  \leq
d$, $m(H,k)=(1+o(1))\frac{n}{d+1}.$
\end{theo}
\begin{theo}
\label{t14}
Any graph with $n$ vertices
and minimum degree $\delta$
contains a spanning subgraph $H$ satisfying
$m(H) \leq (1+o(1)) \left\lceil \frac{n}{\delta+1 } \right\rceil+2.$ 

In addition, if $\delta^{1.24} \geq n $ and 
$n$ is sufficiently large, there is such an $H$ so that
$m(H) \leq \left\lceil \frac{n}{\dd+1} \right\rceil + 2$.
\end{theo}

For any values of $d$ or $\delta$ and $n$, without the assumption
that $n$ is sufficiently large,
we can prove a weaker universal bound showing that there is always 
a spanning subgraph
$H$ with $m(H) $ bounded by $O(n/\delta)$. 
\begin{theo}
\label{t15}
Any $d$-regular graph $G$ with $n$ vertices 
contains a spanning subgraph $H$ satisfying
$m(H) \leq 8\frac{n}{d} +2$.
\end{theo}
\begin{theo}
\label{t16}
Any graph $G$ with $n$ vertices and minimum degree 
$\delta$ 
contains a spanning subgraph $H$ satisfying
$m(H) \leq 16\frac{n}{\delta} +4$.
\end{theo}
We can improve the constants $8$ and $16$ above by a 
more complicated argument,
but since it is clear that these improved constants are not tight we
prefer to present the shorter  proofs of the results above.

Our proofs  combine
some of the ideas used in the earlier
work on the so called irregularity strength of graphs
with techniques from discrepancy theory.
The {\it irregularity strength} $s(G)$ of a graph $G$ with at most
one isolated vertex and no isolated edges 
is the smallest integer
$s$ so that one can assign a positive integer weight between $1$
and $s$ to each edge of $G$ so that for any two distinct vertices $u$
and $v$, the sum of weights of all edges incident with $u$ differs
from the sum of weights of all edges incident with $v$. This notion
was introduced in the 80s in \cite{CJLORS}.
Faudree and Lehel conjectured in \cite{FL}  that 
there exists an absolute constant $C$ so that for every $d$-regular graph
$G$ with $n$ vertices, where $d \geq 2$, $s(G)\leq \frac{n}{d}+C$. The
notion of irregularity strength and in particular the Faudree-Lehel
conjecture
received a considerable amount of attention, 
see e.g. \cite{Le, FGKP, CL, Pr1,Pr2, KKP, MP, Pr3}. 
The theorems above improve some of the results in these 
papers. In particular, Theorems \ref{t13} and \ref{t14} improve 
a result of \cite{FGKP} which implies that any $d$-regular graph
with $n$ vertices contains a spanning subgraph $H$ 
satisfying $m(H) \leq 2n/d$ provided $d^4 \log n \leq n$.
(The result there is stated in terms of assigning weights $1$ and $2$ to
edges, for regular graphs this is equivalent).

Theorems \ref{t14}, \ref{t15} and \ref{t16} 
improve another result of \cite{FGKP}
which implies that any $d$-regular graph with $n \geq 10$ vertices,
where $d \geq 10 \log n$, contains
a spanning subgraph $H$ with $m(H) \leq 48 \log n
\frac{n}{d}$,
as well as a result that for all sufficiently large $n$ any 
$d$-regular 
$G$ contains a spanning $H$ with 
$m(H) \leq 2\frac{n}{\sqrt d}$. They also strengthen a result
in \cite{CL} that shows that any $d$-regular graph
on $n$ vertices contains a spanning
$H$ in which the number of vertices with degrees in any interval
of length  $c_1 \log n$ does not exceed 
$c_2 n \log n/d$ where $c_2>c_1$ are some absolute constants.

Our final results demonstrate a direct 
connection between the irregularity strength of graphs and our
problem here.
\begin{theo}
\label{t17}
Let $G$ 
be a bipartite graph and let $s=s(G)$ be its
irregularity strength. Then $G$ contains a spanning subgraph
$H$ satisfying $m(H) \leq 2s-1$.  If $G$ is regular this can
be improved to $m(H) \leq 2s-3$.
\end{theo}
A similar result, with a somewhat more complicated proof,
holds without the assumption that $G$ 
is bipartite.
\begin{theo}
\label{t18}
Let $G$ 
be a graph and let $s=s(G)$ be its
irregularity strength. Then $G$ contains a spanning subgraph
$H$ satisfying $m(H) \leq 2s$.  If $G$ is regular this can
be improved to $m(H) \leq 2s-2$.
\end{theo}

The rest of the paper contains the proofs as well as a brief final
section suggesting natural versions of the two conjectures that
may be simpler.

\section{Proof of Theorem \ref{t13} and a special case of Theorem 
\ref{t14}}
In this section we prove  Theorem \ref{t13} and describe also a
short proof of Theorem \ref{t14} for the special case that the 
minimum degree $\delta$ satisfies
$\delta^4 =o (n/ \log n)$. The proof of the theorem for larger
$\delta$ requires more work, 
and is presented in Section \ref{sec:proofThm1.4}.

We need several combinatorial and probabilistic lemmas. The first
is the standard estimate of Chernoff for Binomial distributions.
\begin{lemma}[Chernoff's Inequality, c.f., e.g., 
\cite{AS}, Appendix A]
\label{l21}
Let $B(m,p)$ denote the Binomial random variable with parameters
$m$ and $p$, that is, the sum of $m$ independent, identically
distributed Bernoulli random variables, each being $1$ with
probability $p$ and $0$ with probability $1-p$. Then
for every $0<a \leq mp$, 
$\mbox{Prob}(X-mp \geq a) \leq e^{-a^2/3mp}$ and
$\mbox{Prob}(|X-mp| \geq a) \leq 2e^{-a^2/3mp}$.
If $a \geq mp$ then 
$\mbox{Prob}(|X-mp| \geq a) \leq 2e^{-a/3}$.
\end{lemma}
Another result we need is the following, proved (in a slightly
different form) in 
\cite{FGKP}.
\begin{lemma}[\cite{FGKP}]
\label{l22}
Let $G=(V,E)$ be a graph and let $H$ be the spanning random subgraph of
$G$ obtained as follows. For each vertex $v \in V$ let $x(v)$ be a
uniform random weight in $[0,1]$, where all choices are independent.
An edge $uv \in E$ is an edge of $H$ iff $x(u)+x(v) > 1$.
Let $v$ be a vertex of $G$ and suppose its degree in $G$ is $d$.
Then for every $k$, $0 \leq k \leq d$, the probability that the
degree of $v$ in $H$ is $k$ is exactly $\frac{1}{d+1}$.
\end{lemma}
The (simple) proof  given in \cite{FGKP} proceeds by computing the 
corresponding integral. Here is a combinatorial proof, avoiding
this computation. Let $Y=x(v)$ and let $X_1,X_2, \ldots ,X_d$
be the random weights of the neighbors of $v$. Then the random variables
$1-Y, X_1, X_2, \ldots ,X_d$ are i.i.d uniform random variables
in $[0,1]$. By symmetry, $1-Y$ is equally likely to be the 
$k+1$ largest among the variables $1-Y, X_1, \ldots ,X_d$
for all $1 \leq k+1 \leq d+1$, that is, the probability
that $1-Y$ is smaller than exactly $k$ of the variables $X_i$ 
is exactly $1/(d+1)$. The desired results follows as
$1-Y < X_i$ iff $X_i+Y >1$.  \hfill $\Box$
\vspace{0.2cm}

\noindent
We will also use the following well known result of Hajnal
and Szemer\'edi.
\begin{lemma}[\cite{HS}]
\label{l23}
Any graph with $n$ vertices and maximum degree at most $D$ admits
a proper vertex coloring by $D+1$ colors in which every color
class is of size either $\lfloor  n/(D+1) \rfloor$ or
$\lceil n/(D+1) \rceil$.
\end{lemma}

We are now ready to prove Theorem \ref{t13} in the following
explicit form.
\begin{prop}
\label{p24}
Let $G=(V,E)$ be a $d$-regular graph on $n$ vertices. Suppose $0< \eps<1/3$
and assume that the following inequality holds.
\begin{equation}
\label{e21}
(d+1)(d^2+1) 2e^{-\frac{1}{3}\eps^2 \lfloor n/(d^2+1) 
\rfloor \cdot 1/(d+1)} <1.
\end{equation}
Then there is a spanning subgraph $H$ of $G$ so that for every integer $k$,
$0 \leq k \leq d$,
$$
\left|m(H,k)-\frac{n}{d+1}\right| \leq \eps\frac{n}{d+1}.
$$
\end{prop}
\begin{proof}
For each vertex $v  \in V$, let $x(v)$ be a random weight chosen 
uniformly in $[0,1]$, where all  choices are independent.
Let $H$ be the random spanning subgraph of $G$ consisting of all
edges $uv \in E$ that satisfy $x(u)+x(v) >1$. Let $G^{(2)}$ denote
the auxiliary 
graph on the set of vertices $V$ in which two distinct vertices
are adjacent  if and only if their distance in $G$ is either
$1$ or $2$. The maximum degree of $G^{(2)}$ is at most
$d+d(d-1)=d^2$ and hence by Lemma \ref{l23} the set of vertices
$V$ has a partition into $t=d^2+1$  pairwise disjoint subsets
$V_1,V_2, \ldots ,V_t$, where 
$$|V_i|=n_i \in
\{ \lfloor n/(d^2+1) \rfloor, \lceil n/(d^2+1) \rceil\}$$
for all $i$ and each $V_i$ is an independent set in $G^{(2)}$. Note that
this means that the distance in $G$ between
any two distinct vertices $u,v \in V_i$ is at least $2$. As the
degree of each vertex  $v $ of $G$ in $H$ is determined by the
random weights assigned to it and to its neighbors, it follows
that for every fixed $0 \leq k \leq d$, the $n_i$ indicator
random variables $\{Z_{v,k}: v \in V_i\}$ where $Z_{v,k}=1$
iff the degree of $v$ in $H$ is $k$, are mutually independent.
By Lemma \ref{l22} each $Z_{v,k}$ is a Bernoulli random
variable with expectation $1/(d+1)$. For any fixed $k$ as above
it thus follows, by
Lemma \ref{l21} and the assumption inequality (\ref{e21}),
that the probability
that the number of vertices in $V_i$ whose degree in $H$ is $k$
deviates from $n_i/(d+1)$ by at least $\eps n_i/(d+1)$ is smaller than
$\frac{1}{(d^2+1)(d+1)}$. By the union bound over all pairs $V_i,k$, 
with positive probability this does not happen for any $k$ and any
$V_i$. But in this case for every $0 \leq k \leq d$ 
the total number of vertices with degree
$k$ in $H$ deviates from $n/(d+1)$ by less than
$\eps \sum_i n_i/(d+1)=\eps n/(d+1).$ This completes the proof.
\end{proof}
\vspace{0.2cm}

\noindent
{\bf Remark:}\, The proof above is similar to the proof of 
Lemma  7 in \cite{FGKP}. The improved estimate here is obtained
by replacing the application of Azuma's Inequality in \cite{FGKP} by the
argument using the Hajnal-Szemer\'edi 
Theorem (Lemma \ref{l23}), and by an appropriate
different choice of parameters.

By a simple modification of the proof of Proposition \ref{p24}
we next prove the following.
\begin{prop}
\label{p25}
Let $G=(V,E)$ be a graph on $n$ vertices with minimum degree
$\delta$ and maximum degree $\Delta$. Suppose $0<\eps<1/3$
and assume that the following inequality holds.
\begin{equation}
\label{e22}
(\Delta+1)(\delta \Delta+1) e^{-\frac{1}{3}\eps^2 \lfloor 
n/(\delta \Delta+1) \rfloor 
\cdot 1/(\delta+1)} <1.
\end{equation}
Then there is a spanning subgraph $H$ of $G$ so that
$$
m(H) \leq (1+\eps)\frac{n}{\delta+1}.
$$
\end{prop}
\begin{proof}
Start by modifying $G$ to a graph $G'$ obtained by repeatedly
deleting any edge connecting two vertices, both of degrees
larger than  $\delta$, as long as there are such edges. Thus $G'$
is a spanning subgraph of $G$. It has minimum degree $\delta$
and every edge in it has at least one
end-point of degree exactly $\delta$. Let $G'^{(2)}$ denote the
auxiliary graph on the set of vertices $V$ in which two distinct vertices
are adjacent iff they are either adjacent or have a common neighbor
in $G'$. The maximum degree in $G'^{(2)}$ is at most
$$\max\{ \delta+\delta(\Delta-1), \Delta+\Delta(\delta-1)\}=
\delta \Delta.$$
We can now follow the argument in the proof of the previous
proposition, splitting $V$ into $\delta \Delta+1$ nearly equal
pairwise disjoint sets $V_i$, and defining a spanning random subgraph
$H$ of $G'$ (and hence of $G$) using independent random
uniform weights in $[0,1]$ as before. Here for every vertex $v$
and every integer $k$, the probability that the degree of 
$v$ in $H$ is $k$, is at most $1/(\delta+1)$. This, the obvious 
monotonicity, and the fact that the events corresponding to
distinct members of $V_i$ are independent, imply, by Lemma
\ref{l21} and by the assumption inequality 
(\ref{e22}), that the probability
that $V_i$ contains at least $(1+\eps)|V_i|/(\delta+1)$
vertices of degree $k$ is smaller than $\frac{1}{(\delta
\Delta+1)(\Delta+1)}$. The desired result follows from the union
bound, as before.
\end{proof}
\vspace{0.2cm}

\noindent
Similarly, we can prove the following
strengthening of the last proposition.
\begin{prop}
\label{p26}
Let $G=(V,E)$ be a graph with at least $n$ vertices, minimum degree
$\delta$ and maximum degree $\Delta$. Suppose $0< \eps<1/3$. Let
$X \subset V$ be a set of $n$ vertices of $G$
and assume that the inequality (\ref{e22}) holds.
Then there is a spanning subgraph $H$ of $G$ so that for every $k$
the number of vertices in $X$ of degree $k$ in $H$ is at most
$(1+\eps)\frac{n}{\delta+1}$.
\end{prop}
\begin{proof}
The proof is a slight modification of the previous one. Let $G'$ 
be the graph obtained from $G$ as before. 
Let $F$ denote the
auxiliary graph on the set of vertices $X$ in which two distinct vertices
are adjacent iff they are either adjacent or have a common neighbor
in $G'$. The maximum degree in this graph  is at most
$\delta \Delta.$
We can thus follow the argument in the proof of the previous
proposition, splitting $X$ into $\delta \Delta+1$ nearly equal
pairwise disjoint sets $X_i$, and defining a spanning random subgraph
$H$ of $G'$ (and hence of $G$) using the independent random
uniform weights in $[0,1]$ as before. 
\end{proof}

We can now prove the assertion of Theorem \ref{t14} 
provided $\delta^4 =o(n/ \log n)$ in the following 
explicit form.
\begin{prop}
\label{p27}
Let $G=(V,E)$ be a graph on $n$ vertices with minimum degree
$\delta$. Suppose $0< \eps<1/3$. Define $D=\frac{\delta
(\delta+1)}{\eps}$
and assume that the following inequality holds.
\begin{equation}
\label{e23}
(D+1)(\delta D+1) e^{-\frac{1}{3}\eps^2 \lfloor n/(\delta D+1) \rfloor 
\cdot 1/(\delta+1)} <1.
\end{equation}
Then there is a spanning subgraph $H$ of $G$ so that
$$
m(H) \leq (1+2\eps)\frac{n}{\delta+1}.
$$
\end{prop}
\begin{proof}
Let $G=(V,E)$, $\delta$, $\eps$ and $D$ be as above. As in the
previous proofs we start by 
modifying $G$ to a graph $G'$ obtained by repeatedly
deleting any edge connecting two vertices, both of degrees
larger than  $\delta$, as long as there are such edges. Thus $G'$
is a spanning subgraph of $G$; it has minimum degree $\delta$
and the set of all its vertices of degree exceeding $\delta$ is an 
independent set. Let $A$ denote the set of all vertices of degree
$\delta$ in $G'$, $B$ the set of all vertices of  degrees
larger than $\delta$ and at most $D$ in $G'$
(if there are any), and $C$ the set of all vertices of
degree exceeding $D$. Since all edges from the vertices of $C$
lead to vertices of $A$ (as $B \cup C$ is an independent set) it follows,
by double-counting, that $|C| D  < |A| 
\delta \leq n \delta$
and thus $|C| \leq n \delta/D =\eps\frac{n}{\delta+1}.$

If $C=\emptyset$ define $G''=G'$; otherwise let $G''$ be the graph
obtained from $G'$ as follows. For every vertex $v \in C$ of degree
$d (>D)$ replace $v$ by a set $S_v$ of $k_v=\lfloor d/\delta
\rfloor$ new
vertices $v_1,v_2, \ldots ,v_{k_v}$. Split the  set of neighbors
of $v$ in $G'$ (that are all in $A$) into $k_v$ pairwise disjoint
sets $N_1, N_2, \ldots, N_{k_v}$, each of size at least $\delta$
and at most $2 \delta$, and join the vertex $v_i$ to all vertices
in $N_i$ $(1 \leq i \leq k_v)$. Thus $G''$ is obtained by splitting
all vertices of $C$, and there is a clear bijection between the
edges of $G''$ and those of $G'$. Let $X$ be an arbitrary subset of
$n$ vertices of $G''$ containing all vertices in $A \cup B$.
The graph $G''$ has minimum degree
$\delta$ and maximum degree at most $D$; hence by Proposition
\ref{p26} (which can be applied by the assumption inequality (\ref{e23}))
it has a spanning subgraph $H''$ so that no degree is repeated more
than $(1+\eps)\frac{n}{\delta+1}$ times among the vertices of $X$. 
Let $H$ be the spanning subgraph
of $G'$ (and hence of $G$) consisting of exactly the set of edges
of $H''$. The degree of each vertex in $A \cup B$ in $H''$ is the
same as its degree in $H$,
hence $H$ contains at most 
$(1+\eps)\frac{n}{\delta+1}$ vertices of each fixed degree in
$A \cup B$ (which is a subset of $X$). 
We  have no control on the degrees of the vertices of
$C$ in $H$, but their total number is at most $\eps \frac{n}{\delta+1}$.
Therefore $m(H) \leq (1+2\eps)\frac{n}{\delta+1}$, completing the proof.
\end{proof}

\section{Proof of Theorems \ref{t15} and \ref{t16}} 
The main tool in the proofs of Theorems \ref{t15}  and
\ref{t16} is the following result of \cite{ADR}.
A similar application of this result appears in
\cite{Pr1}.
\begin{lemma}[\cite{ADR}]
\label{l31}
Let $G=(V,E)$ be a graph. For each vertex $v \in V$ 
let $\deg(v)$ denote the degree of $v$ in $G$. For each vertex $v$, let
$a(v)$ and $b(v)$ be two non-negative integers satisfying
\begin{equation}
\label{e31}
a(v) \leq  \left \lfloor\frac{\deg(v)}{2}  
\right\rfloor \leq b(v) < \deg(v)
\end{equation} 
and
\begin{equation}
\label{e32}
b(v) \leq \frac{\deg(v)+a(v)}{2}+1~~\mbox{and}~~ b(v) \leq 2a(v)+3.
\end{equation}
Then there is a spanning subgraph $H$ of $G$ so that for every 
vertex $v$ the
degree of $v$ in $H$ lies in the set  
$\{a(v), a(v)+1,b(v),b(v)+1\}$.
\end{lemma}
Theorem \ref{t15} is an easy consequence of this lemma, as we show
next.
\vspace{0.2cm}

\noindent
{\bf Proof of Theorem \ref{t15}:}\, 
Let $G=(V,E)$ be a $d$-regular graph on $n$ vertices. 
Since the assertion is trivial for $d \leq 8$ assume $d>8$.
Put 
$k=\lceil d/4 \rceil$ and split $V$ arbitrarily into
$k$ pairwise disjoint sets of vertices $V_1, \ldots ,V_k$, each of size
at most $\lceil n/k \rceil$.  For each vertex $v \in V_i$ define
$a(v)=\lceil d/2 \rceil -i$ and $b(v)=\lceil d/2 \rceil 
+k -i$. It is easy to check that $\deg(v)=d$ and each such 
$a(v),b(v)$ satisfy (\ref{e31}) and (\ref{e32}). By Lemma \ref{l31}
there is a spanning subgraph $H$ of $G$ in which the degree of every
$v \in V_i$ is in the set 
$$S_i=\{\lceil d/2 \rceil -i, \lceil d/2 \rceil -i+1,
\lceil d/2 \rceil +k -i, \lceil d/2 \rceil +k -i+1\}.$$  
It is easy to check that no integer belongs to more than $2$ of the
sets $S_i$, implying that
$m(H) \leq 2\lceil n/k \rceil < 8n/d+2,$ and completing the
proof. \hfill $\Box$
\vspace{0.2cm}

\noindent
The proof of Theorem \ref{t16} is similar, combining the reasoning
above with one additional argument.
\vspace{0.2cm}

\noindent
{\bf Proof of Theorem \ref{t16}:}\, 
Let $G=(V,E)$ be a graph with $n$ vertices and minimum degree
$\delta$. As the result is trivial for $\delta \leq 16$, assume
$\delta > 16$. Order the vertices of $G$ by degrees, that is,
put $V=\{v_1, v_2, \ldots ,v_n\}$, where the degree of $v_i$ is
$d_i$ and $d_1 \geq d_2 \geq \ldots \geq d_n$. Put
$k=\lceil \delta /4 \rceil$ and split the set of vertices into
$m=\lceil n/k \rceil$ blocks $B_1, B_2, \ldots ,B_m$ of consecutive
vertices in the order above, each (besides possibly the last)
containing $k$ vertices. Thus 
$B_i=\{v_{(i-1)k+1}, v_{(i-1)k+2}, \ldots ,v_{ik}\}$
for all $i<m$ and $B_m=V-\cup_{i<m} B_i$.
Fix a block $B=B_i$; let $w_1, w_2, \ldots , w_k$ denote its
vertices and let $f_1 \geq f_2 \geq \ldots \geq f_k$ be their
degrees (assume now that $B$ is not the last block). For each 
vertex $w_i$ define $a_i=\lceil f_i/2 \rceil -i$,
$b_i=\lceil f_i/2 \rceil +k-i$. For the last block $B_m$
define the numbers $a_i,b_i$ similarly, taking only the first
$|B_m| (\leq k)$ terms defined as above.
Note that the 
sequence $(a_1,a_2, \ldots ,a_k)$ (as well as the possibly shorter one
for the last block) is strictly decreasing,
and so are the sequences $(a_1+1,a_2+1, \ldots ,a_k+1)$,
$(b_1,b_2, \ldots ,b_k)$ and $(b_1+1, b_2+1, \ldots ,b_k+1)$.
Therefore, no integer belongs to more than $4$ of the sets 
$S(w_i)=\{a_i,a_i+1,b_i,b_i+1\}$, $1 \leq i \leq k$. Note also that
the numbers $\deg(v)=f_i, a(v)=a_i, b(v)=b_i$  satisfy (\ref{e31})
and (\ref{e32}). By Lemma \ref{l31} there is a spanning
subgraph $H$ of $G$ in which the degree of every vertex $v$
lies in the corresponding set $S(v)$. Therefore
$m(H) \leq 4m <16\frac{n}{\delta}+4$, completing the proof.
\hfill $\Box$

\section{Proof of Theorem \ref{t17} and \ref{t18}}
\vspace{0.1cm}

\noindent
{\bf Proof of Theorem \ref{t17}:}\,
The proof is based on the simple known fact that the incidence matrix of 
any bipartite  graph is totally unimodular (see, e.g., \cite{Sch},
page 318). Let $G=(V,E)$ be a bipartite graph 
and let $s=s(G)$ be
its irregularity strength. By the definition of $s(G)$  
there is a weight function assigning to each edge $e \in E$ a
weight $w(e)$ which is a positive integer between $1$ and $s$, so
that all the sums $\sum_{e \ni v} w(e)$, $v \in V$  are 
pairwise distinct. Consider the following system of linear
inequalities in the variables $x(e), e \in E$.
$$
0 \leq x(e) \leq 1~~\mbox{for all}~~ e \in E
$$
and 
$$
\left\lfloor \sum_{e \ni v} \frac{w(e)}{s} \right \rfloor \leq 
\sum_{e \ni v} x(e) 
\leq  \left\lceil \sum_{e \ni v} \frac{w(e)}{s} \right \rceil
~~\mbox{for all}~~ v \in V.
$$
This system has a real solution given by
$x(e)=\frac{w(e)}{s}$ for all $e \in E$. Since the $V \times E$
incidence matrix of $G$ is totally unimodular there is an integer
solution as well, namely, a solution in which
$x(e) \in \{0,1\}$ for all $e \in E$. Let $H$ be the spanning
subgraph of $G$ consisting of all edges $e$ with $x(e)=1$.
For each integer $k$ the vertex $v$ can have degree $k$ in $H$ 
only if  $k-1 < \sum_{e \ni v} \frac{w(e)}{s} <k+1$, that is,
only if the integer $\sum_{e \ni v} w(e)$ is strictly
between $s(k-1)$ and $s(k+1)$. As there are only $2s-1$ 
such integers, and the integers $\sum_{e \ni v} w(e)$
are pairwise distinct, it follows that $m(H,k) \leq 2s-1$.

If $G$ is regular one can repeat the above proof replacing
$w(e)$ by $w(e)-1$ for every $e$ and replacing  $s$ by $s-1$.
This completes the proof of Theorem \ref{t17}. \hfill $\Box$

The proof of Theorem \ref{t18} is similar to the last proof, but
requires an additional argument, 
as the incidence matrix of a non-bipartite graph is
not totally unimodular. We thus prove the following lemma. Its proof is
based on some of the techniques of Discrepancy Theory, following
the approach of Beck and Fiala  in \cite{BF}. This lemma will also
be useful in the proof of Theorem \ref{t14} described in the next
section.
\begin{lemma}
\label{l41}
Let $G=(V,E)$ be a graph, and let $z: E \mapsto [0,1]$ be a weight
function assigning to each edge $e \in E$ a real weight $z(e)$ in
$[0,1]$. Then there is a function $x: E \mapsto \{0,1\}$ assigning
to each edge an integer value in $\{0,1\}$ so that for every $v \in
V$
\begin{equation}
\label{e40}
\sum_{e \ni v} z(e)-1 < \sum_{e \ni v} x(e) \leq \sum_{e \ni v} z(e)+ 1.
\end{equation}
\end{lemma}
Note that the deviation of $1$ in this
inequality is tight, as shown by any odd cycle and the function
$z$ assigning weight $1/2$ to each of its edges.
\begin{proof}
We describe an algorithm for generating the required numbers $x(e)$. 
Think of these values as variables. 
During
the algorithm, the variables $x(e)$ will always lie in the continuous
interval $[0,1]$. Call a variable $x(e)$ fixed if $x(e) \in \{0,1\}$,
otherwise call it floating. At the beginning of the algorithm, some
(or all) variables $x(e)$ will possibly be floating, and as the
algorithm proceeds, floating variables will become fixed. Once
fixed, a variable does not change anymore during the algorithm, and
at the end all variables will be fixed. For convenience, call an 
edge $e$ floating iff $x(e)$ is floating.

For each edge $e \in E$, let $y_e$ denote the corresponding column
of the $V \times E$ incidence matrix of $G$, that is, the vector
of length $V$ defined by $y_e(v)=1$ if $v \in e$ and $y_e(v)=0$
otherwise. 

Start the algorithm with $x(e)=z(e)$ for all $e \in E$. As long as
the vectors $y_e$ corresponding to the floating edges $e$ 
(assuming there are such edges) are not linearly
independent over the reals, let 
$\sum_{e \in E'} c_e y_e=0$ be a linear dependence, where $E'$ is a
set of floating edges and $c_e \neq 0$ for all $e \in E'$. 
Note that for any real $\nu$, if we replace $x(e)$ by $x(e)+\nu
c_e$ then the values of the sums
\begin{equation}
\label{e41}
\sum_{e \ni v} x(e)~~\mbox{for all}~~ v \in V 
\end{equation}
stay unchanged. As $\nu$ varies this determines a line
of values of the variables $x(e)$ (in which the only ones that
change are the variables $x(e)$ for $e \in E'$) so that the sums
in (\ref{e41}) stay fixed along the line. By choosing $\nu$
appropriately we can find a point along this line in which
all variables stay in $[0,1]$ and at least one of the floating
variables in $E'$ reaches $0$ or $1$. We now update the variables
$x(e)$ as determined by this point, thus fixing at least one of the
floating variables. Continuing in this manner the algorithm finds 
an assignment of the variables $x(e)$ so that  for each 
$v \in V$, $\sum_{e \ni v} x(e)=\sum_{e \ni v} z(e)$ and 
the set of vectors $y_e, e \in E'$, where $E'$ is the set of
floating edges, is linearly independent. Note that this implies
that the set of edges in each connected component of the graph
$(V,E')$ is either a tree, or contains exactly one cycle, which is
odd.  

As long as there is a connected component consisting of 
floating edges, which is not an odd cycle or a single edge, 
let $V''$ be the set of
all vertices of such a component whose degree in the component
exceeds $1$. Let $E''$ be the set of edges of this component (recall
that all of these edges are floating). Consider the following
system of linear equations.
\begin{equation}
\label{e42}
\sum_{e \ni v} x(e)=\sum_{e \ni v} z(e)~~
\mbox{for all}~~ v \in V''.
\end{equation}
This system is viewed as one in which the only variables are
$x(e)$ for $e \in E''$. The other $x(e)$ appearing
in the system are already fixed, and are thus considered as
constants, and the values $z(e)$ are also constants. It is easy to
check that the number of variables in this system, which is
$|E''|$, exceeds the number of equations, which is the number of
vertices of degree at least $2$ in the component. Therefore there
is a line of solutions, and as before we move to a point on this
line which keeps all variables $x(e)$ in $[0,1]$ and fixes at least
one variable $x(e)$ for some $e \in E''$, shifting it to either $0$
or $1$. Note, crucially, that each of the sums
$\sum_{e \ni v} x(e)$ for $v \in V''$ stays unchanged, but the value
of this sum for vertices of degree $1$ in the component may
change.

Continuing this process we keep reducing the number of floating
edges. When the graph of floating edges contains only connected
components which are odd cycles or isolated edges we finish
by rounding each floating variable $x(e)$ to either $0$ or $1$,
whichever is closer to its current value, where
$x(e)=1/2$ is always rounded to $1$. 
Once this
is done, all variables $x(e)$ are fixed, that is $x(e) \in \{0,1\}$
for all $e$. It remains to show that (\ref{e40}) holds for 
each $v \in V$. To this end note that 
as long as the degree of $v$ in the graph
consisting of all floating edges is at least $2$, and the component
in   which it lies is not an odd cycle, the value of the sum
$\sum_{e \ni v} x(e)$ stays unchanged (and is thus equal exactly to
$\sum_{e \ni v} z(e)$) even after modifying the variables $x(e)$
in the corresponding step of the algorithm. Therefore,
at the first time  the degree of $v$
in this floating graph (the graph of floating edges) becomes $1$,
if this ever happens, the sum $\sum_{e \ni v} x(e)$ is still
exactly $\sum_{e \ni v} z(e)$. Afterwards this sum can change only by
the change in the value of the unique floating edge incident with
it, which is less than $1$ (as this value has been in the open interval
$(0,1)$ and will end being either $0$ or $1$). The only case in
which the final sum  $\sum_{e \ni v} x(e)$ can differ by $1$
from  $\sum_{e \ni v} z(e)$ is if the final step in which all 
floating edges incident with $v$ become fixed is a step in which
the connected component of $v$ in the floating graph is an odd
cycle, $x(e)=1/2$ for both edges of this component incident with
$v$, and both are rounded to the same value $1$. In this case
(\ref{e40}) holds with equality, and in all other cases it holds
with a strict inequality. This completes the proof of the
lemma.
\end{proof}
\vspace{0.2cm}

\noindent
{\bf Proof of Theorem \ref{t18}:}\, 
The proof is similar to that of Theorem \ref{t17}, replacing the
argument using the total unimodularity of the incidence matrix of 
the graph by Lemma \ref{l41}.
Let $G=(V,E)$ be a graph 
let $s=s(G)$ be
its irregularity strength.   Thus
there is a weight function assigning to each edge $e \in E$ a
weight $w(e)$ which is a positive integer between $1$ and $s$, so
that all the sums $\sum_{e \ni v} w(e)$, $v \in V$  are 
pairwise distinct. Define $z: E \mapsto [0,1]$
by $z(e)=w(e)/s$ for each $e \in E$.
By Lemma \ref{l41} there is a function $x: E \mapsto \{0,1\}$
so that for every $v \in V$ (\ref{e40}) holds.
Let $H$ be the spanning
subgraph of $G$ consisting of all edges $e$ with $x(e)=1$.
For each integer $k$ the vertex $v$ can have degree $k$ in $H$ 
only if  
$$
k-1 \leq \sum_{e \ni v} z(e) =\sum_{e \ni v} \frac{w(e)}{s}
< k+1,
$$
that is,
only if the integer $\sum_{e \ni v} w(e)$ is at least
$s(k-1)$ and strictly smaller than $s(k+1)$. As there are only $2s$ 
such integers, and the integers $\sum_{e \ni v} w(e)$
are pairwise distinct, it follows that $m(H,k) \leq 2s$.

If $G$ is regular one can repeat the above proof replacing
$w(e)$ by $w(e)-1$ for every $e$ and replacing  $s$ by $s-1$.
This completes the proof. \hfill $\Box$

\section{Proof of Theorem \ref{t14}}
\label{sec:proofThm1.4}
In this section we describe the proof of Theorem \ref{t14} for all
$\delta $ and $n$ where $n$ is sufficiently large. 
If $\delta=o((n/\log n)^{1/4})$ the assertion of the theorem holds, 
as proved
in Section 2. We thus can and will assume that $\delta$ is larger.
In particular it will be convenient to fix a small $\epsilon>0$ and 
assume that $\dd \geq \ln^{2/\epsilon} n (\ln \ln n)^{1/\epsilon}$.
The  argument here is more
complicated than the one for smaller $\delta$. To simplify the
presentation we omit, throughout the proof,
all floor and ceiling signs whenever these are not crucial
(but leave these signs when this is important). We
further assume whenever this is needed that $n$ is sufficiently
large as a function of $\epsilon$. The explicit version of the 
theorem we prove here is the following. 
\begin{theo}
\label{thm:delta}
Fix $\epsilon \in (0,1/4)$. Every graph $G$ with $n$ vertices
and minimum degree $\dd$ with $\dd \geq \ln^{2/\epsilon} 
n (\ln \ln n)^{1/\epsilon}$ 
and $n$
sufficiently large in terms of $\epsilon$ contains a spanning
subgraph $H$ satisfying
$$
m(H) <  \left\lceil \frac{n}{\dd} + \frac{5\sqrt{(n/\dd)(\ln n)}}
{\dd^{1/4}} \right\rceil + \left\lfloor \frac{2016 n \ln
n \ln \ln n}{\dd^{1+\epsilon}} \right\rfloor + 1.
$$

In particular, when $\delta^{1+\epsilon} > 2016 n \ln n \ln \ln n$, 
and $n$ is sufficiently large, then  
$$
m(H) \leq \lceil n/(\dd+1) \rceil + 2.
$$
\end{theo} 
The constants $5, 2016$ and the assumption 
$\delta^{1+\epsilon} > 2016 n \ln n \ln \ln n$ can be improved,
but as this will not lead to any significant change in
the asymptotic  statement given in
Theorem \ref{t14} it is convenient to prove the result as stated
above.

In the proof we  assign
binary weights to the edges of the graph $G$, where weight
one corresponds to edges in $H$ and zero to non-edges. The weight
of a vertex will always be the sum of weights of the 
edges incident to it.
We use $\deg(v)$ to denote the degree of the vertex $v$ in $G$. By
assumption $\deg(v) \geq \dd$ for every $v$.

Put $\sstar = \dd^{1/2+\epsilon}, k = \dd^{1/2-\epsilon}/\ln \ln n$. By
our assumption on $\dd$ we have $k \geq \ln^2 n$. We will assume
that both $\sstar$ and $k$ are integers, and $\dd-\sstar$ is divisible
by $\lfloor \sqrt{\dd}\rfloor$ \footnote{There is always a value of
$\sstar$ in the interval $[\lfloor \dd^{1/2+\epsilon} \rfloor,\lfloor
\dd^{1/2+\epsilon} \rfloor + \sqrt{\dd}]$  such that $\dd-\sstar$
is divisible by $\lfloor \sqrt{\dd}\rfloor$. When $\epsilon$ is fixed
and $n$ sufficiently large, such value of $\sstar$ is asymptotically
$\dd^{1/2+\epsilon}$.}.  We start by partitioning the vertices randomly
into a big set $B$ and a small set $S$, where each is further partitioned
into $B = B_1 \dots \cup B_{\dd-\sstar}$ and $S =  S_1 \cup \dots
\cup S_{k}$. This random partition is achieved in the following way.
Let $X_v$, $v \in V(G)$ be i.i.d.~ uniform random variables $X_v\sim
U[0,1]$.  For each integer $1 \leq i \leq \dd-\sstar$, if $X_v \in
[\frac{i-1}{\dd}, \frac{i}{\dd })$, then place $v$ in $B_i$.  For each
integer $1 \leq j \leq k-1$, if $X_v \in \left[\frac{\dd-\sstar}{\dd} +
\frac{(j-1)\sstar}{\dd k},\frac{\dd-\sstar}{\dd} + \frac{j \sstar}{\dd
k}\right)$, place $v$ in $S_j$; if $X_v \in \left[\frac{\dd-\sstar}{\dd}
+ \frac{(k-1)\sstar}{\dd k},1\right]$, place $v$ in $S_k$.

The weight assignment will be done in three steps. The first two steps only
concern edges in $B$ and between $B$ and $S$. The last step only concerns
edges within $S$.  We will randomly label some edges between $S$ and $B$
to be {\it active} and {\it removable}. Active edges denote the edges
between $S$ and $B$ that will be assigned weight one in Step 1, and 
active and removable edges denote ones 
whose weights can be modified back to zero in Step 2. For
each $1 \leq i \leq k$, vertex $v \in S_i$ and its neighbor $u\in B$,
the edge $uv$ is active randomly and 
independently with probability ${\frac{\dd - 4 \sstar
i }{\dd-\sstar}}$. It is  removable randomly and independently with
probability $\frac{32 \dd  \ln n}{\sstar \sqrt{\deg(u)}}$.

The next lemma shows that the quantities we 
care about in $G$ are not far from
their expected values with high probability.

\begin{lemma} 
\label{lem:main}
Let $\epsilon\in(0,1/4)$. Suppose $n$ is sufficiently large in terms
of $\epsilon$ and assume that 
$\dd^{\epsilon} \geq \ln n \ln \ln n$.
Let $h:[n]\times [\dd
- \sstar] \to \mathbb{R}$ be a function  $h(d, i) = c_1(d) i +c_2 d +
c_3 \sqrt{d}  + c_4$ where $c_1(d) \geq 1$ for all $d \geq \dd$ and $c_2, c_3, c_4
\in \mathbb{R}$. Then, with probability at least $1 - 7/n^2$, 
the following statements hold simultaneously with the random
choices  described above.
\begin{enumerate}[(i)]
\item 
\label{item:Bi2} 
For any integer $0 \leq j \leq n-1$, the number of vertices $v \in B$
satisfying $h(\deg(v), Z(v)) \in [j, j+\lfloor \sqrt{\delta} \rfloor)$
is at most $ \lfloor \sqrt{\delta} \rfloor \frac{n}{\dd} +4 \sqrt{\frac{n
}{\dd} \cdot \sqrt{\delta} \ln n}$, where $Z(v) \in [\dd - \sstar]$
is the random variable satisfying $v \in B_{Z(v)}$.
\item 
\label{item:vtoS} 
For any vertex $v$, its degree to $S$ 
is in the interval $[0.5 \sstar \deg(v)/ \dd, 1.5 \sstar \deg(v)/ \dd]$. 
\item
\label{item:vBi} 
For each $1 \leq i \leq \dd - \sstar$ and for each vertex $v \in B_i$, its
degree to $\{\bigcup B_j, \dd-\sstar - i + 1 \leq  j \leq \dd - \sstar\}$
is in 
$\left[\frac{i\deg(v)}{\dd} -12\sqrt{\frac{i\deg(v)}{\dd}} \ln n,
\ \frac{i\deg(v)}{\dd} + 12\sqrt{\frac{i\deg(v)}{\dd}} \ln n\right]$.
\item 
\label{item:vStoBactive}
For each $1 \leq i \leq k$ and each vertex $v \in S_i$, the number
of edges between $v$ and $B$ that are active is in the interval
$$\left[\frac{(\dd - 4\sstar i)\deg(v)}{\dd} - \sqrt{\deg(v) }\ln n, 
\frac{(\dd - 4\sstar i)\deg(v)}{\dd} +\sqrt{\deg(v)} \ln
n\right].$$  
The number of edges between 
$v \in S$ and $B$ that are both active and removable
is at most $ \frac{33(\dd - 4\sstar i)\deg(v) \ln n}{\sqrt{\dd} \sstar}$.
\item \label{item:vBtoSiactive} 
For each $1 \leq i \leq k$ and each $u \in B$,  the number of
edges between $u$ and $S_i$ that are active is in the interval
$$\left[\frac{\sstar \deg(u)}{\dd k} \cdot \frac{\dd - 4\sstar
i}{\dd - \sstar} - \sqrt{\frac{\deg(u) \sstar}{\dd k}} \ln n, 
\frac{\sstar \deg(u)}{\dd k} \cdot \frac{\dd - 4\sstar i}{\dd - \sstar}
+ \sqrt{\frac{\deg(u) \sstar}{\dd k}} \ln n\right].$$  
The number of
edges between $u \in B$ and $S$  
that are both active and removable is at least
${27\sqrt{\deg(u)} \ln n}$.
\end{enumerate}
\end{lemma}
\begin{proof}
We first prove (\ref{item:Bi2}). 
Given $j$ and $v$,  since $c_1(\deg(v)) \geq 1$ as $\deg(v) \geq \dd$, there is 
at most one integer $1 \leq i 
\leq \dd - \sstar$ such that $h(\deg(v), i)  = c_1(\deg(v)) i  + c_2 \cdot \deg(v) + c_3
{\sqrt{\deg(v)}}  + c_4 \in [j, j+1)$. Thus each vertex independently has
probability at most $\lfloor \sqrt{\delta} \rfloor / \dd$ 
to satisfy $h(\deg(v),
Z(v)) \in [j, j+\lfloor \sqrt{\delta} \rfloor)$. By 
Chernoff's Inequality (Lemma \ref{l21}) and 
a union bound
over $0 \leq j \leq n-1$ and $v$, the probability that (\ref{item:Bi2})
is violated is at most $ n^2 e^{-(16(n/\delta) 
(\sqrt{\delta}) \ln n)/(3 (n/\dd)(\sqrt{\delta}))}
< n^2 e^{-4 \ln n} = 1/n^2$.

To prove (\ref{item:vtoS}), note that for each vertex $v$, each of
its neighbors independently has probability $\frac{\sstar}{\dd}$
to be in $S$. Therefore its expected degree in $S$ is
$\frac{\deg(v)\sstar}{\dd}$. By Chernoff's Inequality and 
a union bound over $v$,
(\ref{item:vtoS}) is violated with probability at most $2ne^{- 0.25
\deg(v)\sstar /(3\dd)} < 1/n^2$.

To prove (\ref{item:vBi}), note that for each $1 \leq i \leq
\dd-\sstar$,
each neighbor of $v$ independently has probability $i/\dd$
to be in $B_{\dd-\sstar-i+1} \cup \dots \cup B_{\dd-\sstar}$, and thus
the expected number of its neighbors in $B_{\dd-\sstar-i+1} \cup \dots
\cup B_{\dd-\sstar}$ is $\frac{i\deg(v)}{\dd}$.  By Chernoff's Inequality,
for any positive value $\mu$, given $v \in B_i$, the probability that
(\ref{item:vBi}) is violated is at most $2e^{-\mu^{2}/(3\max(i\deg(v)/\dd,
\mu))}$. Plugging in $\mu = 12\sqrt{i \deg(v)/\dd} \ln n$ and
noting
that $\max(i\deg(v)/\dd, \mu)) \leq 12 \frac{i \deg(v)}{\dd} \ln n$,
the probability that (\ref{item:vBi}) is violated for $v \in B_i$ is at
most $2 e^{- 12 \ln n / 3}= 2/n^4$. By a union bound over all vertices
the probability that (\ref{item:vBi})
is violated is much smaller than  $1/n^2$.

Similarly we can prove (\ref{item:vStoBactive}). Given $v \in S_i$,
each edge incident to $v$ independently has probability $\frac{\dd -
\sstar}{\dd} \cdot \frac{\dd - 4 \sstar i}{\dd - \sstar} = \frac{\dd -
4\sstar i}{\dd}$ to be active. Thus the expected number of active edges
incident to $v \in S_i$ is   $\frac{\deg(v)(\dd - 4\sstar i)}{\dd}$.
Again by  
Chernoff's Inequality and a union bound over $v$, the first statement in
(\ref{item:vStoBactive}) is violated with probability at most $n 2e^{-
\deg(v) \ln^2 n / (3(\dd - 4\sstar i)\deg(v)/\dd  )} < n2e^{-\ln^2n/3}<
1/n^2$.  Similarly, for a neighbor $u$ of $v \in S_i$, the edge $(v, u)$
randomly and independently 
has probability $\frac{\dd - 4\sstar i}{\dd} \frac{32\dd \ln
n}{\sstar \sqrt{\deg(u)}} \leq \frac{32 (\dd - 4\sstar i)  \ln n}{\sstar
\sqrt{\dd}}$ to be both active and removable. By 
Chernoff's Inequality and a union bound over $v$ the probability
that the second statement is violated is much smaller than $1/n^2$.

(\ref{item:vBtoSiactive}) is proved in almost the same way.
Fix $1 \leq i \leq k$ and $u \in B$. Each edge $uv$ 
independently has probability
$\frac{\sstar}{k\dd} \cdot \frac{\dd - 4\sstar i}{\dd - \sstar}$ to
be active and satisfy  $v \in S_i$; 
and it has probability $\frac{\sstar}{k\dd}
\cdot \frac{\dd - 4\sstar i}{\dd - \sstar} \cdot \frac{32\dd \ln n}{\sstar
\sqrt{\deg(u)}} = \frac{32 \ln n}{k} \cdot \frac{\dd - 4\sstar i}{(\dd -
\sstar)\sqrt{\deg(u)}} > \frac{30 \ln n}{k \sqrt{\deg(u)}}$ to be both
active and removable and incident to $S_i$. Thus for any $u \in B$ 
the expected number of edges $uv$ with  $v \in S$ 
which are both active and 
removable is at least ${30\ln n}{\sqrt{\deg(u)}}$. 
Applying
Chernoff's Inequality and a union bound it follows that the
probability the statement fails is much smaller than $2/n^2$.
\end{proof}

Therefore, with probability at least $1- 7/n^2$ all assertions of Lemma 
\ref{lem:main} hold, where the function 
$h(\deg(v), i)$ in (\ref{item:Bi2}) is 
\[h_B(\deg(v),i) = 
\frac{i\deg(v)}{\dd}+ \frac{\sstar \deg(v)}{\dd }\frac{\dd - 
2\sstar (k+1)}{\dd - \sstar} -  13\sqrt{\deg(v)} \ln n.\]
Since $\deg(v) \geq \delta$, $k \sstar 
= \dd/\ln \ln n$ and by the lower bounds on $\dd$ in the assumption, 
it is easy to see that $h_B(d,i) > \sstar / 2 > \sqrt{\delta}$. 
Note that $h_B$ satisfies the 
requirement of $h(d,i)$ in Lemma \ref{lem:main}. 
We can now proceed assigning weights in $\{0,1\}$ 
to the edges in $G$ in three steps. 

In Step 1, we assign the following edges weight one: (1) for all $1 \leq
i \leq \dd - \sstar$, all the edges between $B_i$ and $\{ \bigcup_j B_j,
\dd - \sstar - i + 1 \leq j \leq \dd - \sstar\}$; (2) all the active
edges between $B$ and $S$.

In Step 2, the goal is to ensure that each vertex weight appears 
in at most
\begin{equation}
\lceil n/\dd +
5\sqrt{n/\dd} \sqrt{\ln n}/\delta^{1/4}\rceil \label{eq:step2}
\end{equation}
vertices in $B$. This is achieved by
making two modifications.
First ensure that each vertex $v$ in $B_i$ has weight exactly
$\lfloor h_B(\deg(v),i) \rfloor$.  By Lemma \ref{lem:main} applied with
$h_B(d,i)$, with probability at least $1-7/n^2$ after Step 1, 
for each $1 \leq i \leq \dd - \sstar$, 
the weight of $v \in B_i$ deviates from
\[ \frac{i\deg(v)}{\dd}+ \sum_{j=1}^k \frac{\sstar \deg(v)}{\dd k} \cdot
\frac{\dd - 4\sstar j}{\dd - \sstar} = h_B(\deg(v),i) + 13\sqrt{\deg(v)}
\ln n \]
by at most ${k\sqrt{\frac{\deg(v) \sstar }{\dd k}} \ln n +
12\sqrt{\deg(v)} \ln n < 13\sqrt{\deg(v)} \ln n}$ by Lemma
\ref{lem:main} (\ref{item:vBi}) and the first statement in
(\ref{item:vBtoSiactive}).  Thus it is possible
to transform  the weight of $v$ to  exactly $\lfloor
h_B(\deg(v),i) \rfloor$ by reducing the weights 
of at most
$26\sqrt{\deg(v)} \ln n  +1 <  26.5\sqrt{\deg(v)} \ln n$ 
(active and removable) edges from  $v$ to $S$
from one to zero.

Suppose this first modification is possible, in the second modification,
by Lemma \ref{lem:main} (\ref{item:Bi2}) and the fact 
that $h_B(\deg(v),i) > \sqrt{\delta}$ for each vertex $v$ and 
$1\leq i \leq \dd-\sstar$ and that $\dd - \sstar$ is divisible 
by $\lfloor \sqrt{\dd} \rfloor$, we can further reduce the
weights of at most $2(\sqrt{\delta})$ edges between each $v \in B$ 
and $S$ ensuring that each
integer vertex weight appears in at most 
$\left\lceil \left( \lfloor \sqrt{\delta} \rfloor
n/\dd + 4  \sqrt{(n/\dd)(\sqrt{\delta}) \ln n} \right)/ 
\lfloor \sqrt{\delta} \rfloor \right\rceil \leq  \lceil n/\dd +
5\sqrt{n/\dd} \sqrt{\ln n}/\delta^{1/4}\rceil$ vertices in $B$, 
as desired. Indeed, this can be done by considering, for any fixed
admissible $j>1$, all vertices whose weight  after the first
modification lies in $[(j-1)( \lfloor \sqrt{\delta} \rfloor),
j( \lfloor \sqrt{\delta} \rfloor))$. Their weights can be reduced
and distributed uniformly among the possible weights
in the interval $[(j-2)( \lfloor \sqrt{\delta} \rfloor),
(j-1)( \lfloor \sqrt{\delta} \rfloor))$.

It is not difficult to check that these two modifications can be
accomplished by reducing only the weights  of some
edges which are both active and removable.
Indeed, for every vertex $v \in B$ it is only needed to reduce its weight 
by at most  $26.5\sqrt{\deg(v)} \ln n+ 2 \sqrt{\delta} < 27\sqrt{\deg(v)} 
\ln n$.
By Lemma \ref{lem:main}  (\ref{item:vBtoSiactive}), the number of
edges between $v$ and $S$ which are both active and removable is at
least $27 \sqrt{\deg(v)}\ln n$, and as all active edges between $B$
and $S$ have weight one
prior to Step 2 there are enough edges whose weights can be reduced 
from one to zero to allow the two modifications.

In Step 3, we will only adjust the weights of edges within $S$ 
to ensure 
that each weight appears in at most 
$ \lfloor \frac{2016 n \ln n \ln \ln n}{\dd^{1+\epsilon}} 
\rfloor + 1$ vertices in $S$. We first use
a method developed in a paper in preparation by the second author and
J. Przyby{\l}o~\cite{PW} to identify which vertices in $S$ might 
have the same weight.
For each vertex $v \in S$, we will define a set $L(v)$ such that $v,u \in
S$ cannot have the same weight at the end of Step 3 if $u \notin L(v)$. We
will then show that  with high probability all sets $L(v)$ will
not be large.

To start, we relax the problem where the weight of each edge in $S$
can be any real number in $[0,1]$.  We first analyze the range of weight
for each $v \in S$ after adjusting weights in $S$.

By Lemma \ref{lem:main}  (\ref{item:vStoBactive}), after Step 2,
since $\sqrt{\deg(v)} \ln n \leq \deg(v)\ln n /\sqrt{\dd}$, the weight
of $v \in S_i$ is at least the number of edges incident to $v$ 
which are active but not removable, which is bounded below by
\begin{align*}
& \frac{(\dd - 4\sstar i)\deg(v)}{\dd} -
\frac{\deg(v)\ln n}{\sqrt{\dd}} - \frac{33(\dd - 4\sstar i)\deg(v) \ln
n}{\sqrt{\dd}\sstar}  \\
= & \deg(v) \left(\left(1-\frac{4\sstar i}{\dd}\right) \left(1
- \frac{33\ln n}{\dd^{\epsilon}}\right)- \frac{\ln n}{\sqrt{\dd}}
\right) \geq  \deg(v) \left(1-\frac{4\sstar i}{\dd} - \frac{34\ln
n}{\dd^{\epsilon}} \right).
\end{align*}  
This is also a lower bound on the weight of $v \in S_i$ after
Step 3.
By Lemma \ref{lem:main} (\ref{item:vtoS}), the additional weight each
vertex $v \in S$ can gain in Step 3 is at most $\deg_S(v) \cdot 1 \leq
1.5 \deg(v)\sstar/\dd$. Again together with  Lemma \ref{lem:main}
(\ref{item:vStoBactive}), the weight of $v \in S_i$ after Step 3 is
at most
$\left(\frac{(\dd - 4\sstar i)\deg(v)}{\dd}+\frac{\deg(v)\ln
n}{\sqrt{\dd}} \right) +1.5 \deg(v)\sstar/\dd = \deg(v) \left( 1 -
\frac{4\sstar i}{\dd} + \frac{\ln n}{\sqrt{\dd}}  + \frac{1.5\sstar}{\dd}
\right) < \deg(v) \left( 1 - \frac{4\sstar i}{\dd} + \frac{3\sstar}{{\dd}}
\right)$. In summary, the weight of $v \in S_i$ after assigning arbitrary
weights in $[0,1]$ to edges in $S$ is always in the interval
\begin{equation}
 I_{v, i}= 
 \left[
 \deg(v) \left(1-\frac{4\sstar i}{\dd} 
- \frac{34\ln n}{\dd^{\epsilon}} \right) ,   
 \deg(v) \left( 1 - \frac{4\sstar i}{\dd}  + \frac{3\sstar}{\dd}  \right)
 \right]. \label{eq:intervalSi2}
  \end{equation}
Therefore, vertex $v \in S_i$ and $u\in S_j$ can have the same 
weight after Step 3  only if $I_{v,i} \cap I_{u,j} \neq \emptyset$, 
which is equivalent to
\begin{align}
\deg(v) \left(1-\frac{4\sstar i}{\dd} 
- \frac{34\ln n}{\dd^{\epsilon}} \right) 
\leq &
\deg(u) \left( 1 - \frac{4\sstar j}{\dd} 
+  \frac{3\sstar}{\dd}  \right);  \text{ and } \label{eq:ineq1}
\\
\deg(u) \left(1-\frac{4\sstar j}{\dd}
- \frac{34\ln n}{\dd^{\epsilon}} \right) \leq &
\deg(v) \left( 1 - \frac{4\sstar i}{\dd}  
+ \frac{3\sstar}{\dd}  \right).  \label{eq:ineq2}
\end{align}
Let $u \in L(v)$ if and only if $u \neq v$, 
$u \in S$, and  both (\ref{eq:ineq1})
and (\ref{eq:ineq2})  hold if $u \in S_j$. 
Clearly $u \notin L(v)$ implies the distinct vertices $u,v \in S$
have distinct weights.

\begin{claim}
\label{claim:sizeLv}
With probability at least $1-1/n^2$, 
$|L(v)| \leq 
\frac{42n\deg(v) \ln n}{k \dd^{1+\epsilon} }$ for all $v \in S$. 
\end{claim}
\begin{proof}
Given $v \in S_i$ we bound the number of vertices  $u$ in $L(v)$
by bounding  the number of vertices $u$ in all sets $S_j$ 
where $j$ satisfies both
inequalities (\ref{eq:ineq1}) and (\ref{eq:ineq2}). These two inequalities
together imply
\[
    \frac{\deg(v)}{\deg(u)} \left(1-\frac{4\sstar i}{\dd} 
- \frac{34\ln n}{\dd^{\epsilon}} \right) -  \frac{3\sstar}{\dd} 
    \leq 
    1-\frac{4\sstar j}{\dd} 
    \leq 
    \frac{\deg(v)}{\deg(u)} \left( 1 - \frac{4\sstar i}{\dd}  
+ \frac{3\sstar}{\dd}  \right) + \frac{34\ln n}{\dd^{\epsilon}}.
\]
This means the value of $\frac{4\sstar j}{ \dd}$ 
can only lie in an interval of length 
$\frac{\deg(v)}{\deg(u)}\left(\frac{3\sstar}{\dd} +  \frac{34\ln
n}{\dd^{\epsilon}}  \right) + \frac{34\ln n}{\dd^{\epsilon}}
+\frac{3\sstar}{\dd}  \leq \frac{2\deg(v)}{\dd}\left(\frac{3\sstar}{\dd}
+  \frac{34\ln n}{\dd^{\epsilon}}  \right) \leq \frac{2\deg(v)}{\dd}
\frac{37\ln n}{\dd^{\epsilon}} $ where the last inequality 
uses $\epsilon < 1/4$. This implies $j$ can only lie in an
interval of length at most $\frac{\dd}{4\sstar} \frac{2\deg(v)}{\dd}
\frac{37\ln n}{\dd^{\epsilon}} = \frac{\deg(v)}{2\sstar} \frac{37\ln
n}{\dd^{\epsilon}}$.  

Since with probability $\frac{\sstar}{k \dd}$, the vertex $u$ 
lies in $S_j$
for any given $j$, the probability 
that $u$ satisfies $u \in S_j$ for some $j$ with $I_{v,i}
\cap I_{u,j} \neq \emptyset$ is at most 
$$\frac{\sstar}{k \dd}
\left(\frac{\deg(v)}{2\sstar } \frac{37\ln
n}{\dd^{\epsilon}}+1\right) \leq \frac{\sstar}{k \dd
}\frac{\deg(v)}{\sstar } \frac{37\ln n}{\dd^{\epsilon}} =
\frac{37\deg(v) \ln n}{k \dd^{1+\epsilon} }. $$ 
The first inequality uses the fact that $\epsilon < 1/4$. 
Thus
$\mathbb{E}[|L(v)|] \leq n\frac{37\deg(v) \ln n}{k
\dd^{1+\epsilon} } $.  Since the events for different vertices $u$ 
are independent, by
Chernoff's Inequality the probability that $|L(v)| \geq n\frac{42\deg(v)
\ln n}{k \dd^{1+\epsilon} }$ is at most $n^{-4}$. By a union bound
over $v$, the desired result follows.
\end{proof}

We are now ready to adjust the weights of edges in $S$.  First we show
there is a desired weighting with edges in $S$ having fractional weights
in $\{0, 1/4, 1/2, 3/4, 1\}$.
\begin{claim}
\label{claim:Sfractional}
With probability at least $1-8/n^2$ one
can assign each edge in $S$
a weight in $\{0, 1/4, 1/2, 3/4, 1\}$ such that for each vertex $v \in
S$, the number of vertices in $L(v)$ whose weight (including the
weight to $B$) differs from that of $v$ by strictly less than $11/4$  
is at most $\lfloor \frac{2016 n \ln n \ln \ln n}{\dd^{1+\epsilon}}
\rfloor
$. 

In particular, if $\delta^{1+\epsilon} > 2016 n \ln n \ln \ln n$, for
any two vertices $v, u \in S$ where $u \in L(v)$, the difference between
the weights of $v$ and $u$ is at least $11/4$.
\end{claim}
\begin{proof}
We use a modified version of the algorithm  by 
Kalkowski, Karo\'nski, and Pfender~\cite{KKP}.
All edge weights in $S$ are initialized to be $1/2$. 

Order the vertices of $S$ arbitrarily  as $v_1, v_2, \dots$ and
process them
sequentially starting from $v_1$. When processing
$v_i$, we will find a set $\Lambda_{v_i}$ of the form $\{\frac{12a}{4},
\frac{12a+1}{4}\}$ for some $a \in \mathbb{Z}$, such that throughout
the later stages of the algorithm, $\Lambda_{v_i}$ will stay unchanged and
the weight of $v_i$ will always stay in  $\Lambda_{v_i}$. 
Suppose we are processing $v_i$ for $i \geq 1$. For each forward
edge, i.e., edge $v_i v_j$ where $j > i$ if exists, we allow to 
change the edge
weight by increasing it by $0$ or $1/4$; 
for each  backward edge $v_i v_j$
where $j < i$ if exists, we allow to change the weight  
by adding an element of $\{-1/4, 0, 1/4\}$, where
if the current weight of $v_j$ is the maximum value  in $\Lambda_{v_j}$,
we can only change this backward edge by adding a member of  
$\{-1/4, 0\}$, whereas if the
current weight of $v_j$ is the minimum value  in $\Lambda_{v_j}$, we can
only change this backward edge by adding a member of
$\{0, 1/4\}$. This rule guarantees
 the weight of $v_j$ which has been 
processed always stays in $\Lambda_{v_j}$. Furthermore, by all
combinations of the allowable changes, the weight of $v_i$ can achieve
any value in an arithmetic progression $P_i$ with common difference $1/4$
and of length $\deg_S(v_i)$.
In addition, by our constraints on the structure of 
the sets $\Lambda_{v_i}$,
a vertex $v_i$ has weight in $\Lambda_{v_i} 
= \{\frac{12a}{4}, \frac{12a+1}{4}\}$ if and only if $v_i$ has weight 
in $J_{v_i} = \{\frac{12a}{4},
\frac{12a+1}{4}, \frac{12a+2}{4}, \dots, \frac{12a+11}{4}\}$. 
Thus there must be a set $ \{\frac{12b}{4},
\frac{12b+1}{4}, \frac{12b+2}{4}, \dots, \frac{12b+11}{4}\} 
\subset P_i$ for some $b \in \mathbb{Z}$ which is shared
 by at most $\lfloor |L(v)| / (( |P_i| -22)/12) \rfloor $ 
sets $J_{v_j}$ for $v_j \in
 L(v)$ and $j < i$. 
Fix such a set $\{\frac{12b}{4},
\frac{12b+1}{4}, \frac{12b+2}{4}, \dots, \frac{12b+11}{4}\} 
\subset P_i$ as $J_{v_i}$ and then ensure that the weight of $v_i$ 
lies in $\{ \frac{12b}{4}, \frac{12b+1}{4}\} = \Lambda_{v_i}$
by adjusting the weights of
forward and backward edges appropriately, and then continue to 
$v_{i+1}$.  
By Claim \ref{claim:sizeLv} and Lemma \ref{lem:main}
(\ref{item:vtoS}) which implies that $|P_i| \geq 0.5 \sstar \deg(v)/\dd$,
\[ |L(v)| / (( |P_i| -22)/12)  <  \frac{12 \cdot \frac{42n\deg(v) \ln n}{k
\dd^{1+\epsilon} } }{ 0.25 \deg(v) \sstar / \dd} \leq \frac{48\cdot
42 n}{k  } \frac{\ln n}{\dd^{\epsilon} \sstar } = \frac{2016 n \ln
n \ln \ln n}{\dd^{1+\epsilon}} \]
where the equality is by plugging in {$k\sstar = \dd/ \ln \ln n$}. 
Therefore we have shown that each set $J_{v_i}$ can
be shared by at most $\lfloor \frac{2016  n \ln n \ln \ln n}
{\dd^{1+\epsilon}} \rfloor$ other
$J_{v_j}$ for $v_j \in L(v_i)$. 
Furthermore, if $J_v$ is different from $J_u$ which implies $J_v$
is disjoint from $J_u$, then since the weight of $v$ is in $\Lambda_v
\subset J_v$ and the weight of $u$ is in $\Lambda_u \subset J_u$, 
the difference between
the weights of $u$ and $v$ is at least $11/4$.
Lastly, notice that each edge 
changes its weight at most twice 
(once as a forward edge and once as a backward
edge), so all edge weights in $S$ stay in 
$\{0, 1/4, 1/2, 3/4, 1\}$. Therefore the first statement holds. 
The second statement holds by noticing that when 
$\delta^{1+\epsilon} > 2016 n \ln n \ln \ln n$, then 
$\lfloor \frac{2016  n \ln n \ln \ln n}{\dd^{1+\epsilon}} \rfloor = 0$. 
\end{proof}

Suppose $\dd^{\epsilon} \geq \ln^2 n \ln \ln n$. 
We are now ready to finish the
construction and the proof.
Suppose $z(e)$ are the current weights of edges $e$ in $E(G)$ where when
$e \in E(S)$, $z(e) \in \{0, 1/4, 1/2, 3/4, 1\}$ and when $e \notin E(S)$,
$z(e) \in \{0,1\}$. 
We now show that we can change the edge weights in $S$ to be
in $\{0,1\}$ so that each weight is  shared by 
at most $ \lfloor \frac{2016 n \ln n \ln \ln n}{\dd^{1+\epsilon}} 
\rfloor + 1$ vertices
in $S$.  To achieve this, we apply 
Lemma \ref{l41} to the induced subgraph on $S$ to conclude
that  there is a binary weighting
$x: E(G) \to \{0,1\}$ such that $x(e) = z(e)$ for $e \notin E(S)$,
and for each $v\in S$, 
\begin{equation}
\sum_{e \ni v}
z(e) - 1 < \sum_{e \ni v} x(e) \leq \sum_{e \ni v} z(e) + 1. \label{eq:wz}
\end{equation}
We now bound the number of vertices in $S$  sharing the same weight. 
Given $v \in S$, if a different vertex $u\in S$ satisfies 
$\sum_{e \ni v} x(e)= \sum_{e \ni u} x(e)$, then $u \in L(v)$. 
Furthermore, by  the triangle inequality and (\ref{eq:wz}), 
$$
0 = \left|\sum_{e \ni v} x(e) - \sum_{e \ni u} x(e)\right| \geq 
\left|\sum_{e \ni v} z(e) - \sum_{e \ni u} z(e)\right| - 2, 
$$
which implies $\left|\sum_{e \ni v} z(e) - \sum_{e \ni u} z(e)\right|
\leq  2 < 11/4$. By Claim \ref{claim:Sfractional}, there are at most
$\lfloor \frac{2016 n \ln n \ln \ln n}{\dd^{1+\epsilon}} \rfloor$
different $u \in L(v)$ with  $\left| \sum_{e \ni v} z(e) - \sum_{e
\ni
u} z(e)\right| < 11/4$. Thus each weight with respect to $x$ is shared
by at most   $\lfloor \frac{2016 n \ln n \ln \ln n}{\dd^{1+\epsilon}}
\rfloor + 1$ vertices in $S$, as desired.

We have shown in (\ref{eq:step2}) in  Step 2 that 
the number of vertices in $B$ with the
same weight is at most $\lceil n/\dd +
5\sqrt{n/\dd} \sqrt{\ln n}/\delta^{1/4}\rceil$, and note that  weights
of vertices in $B$ do not change after Step 2.  Therefore we have shown
that there is a spanning 
subgraph $H$ of $G$ (corresponding to the edges with
$x(e)=1$) satisfying $m(H) \leq \left\lceil n/\dd + 5 \sqrt{n/\dd}
\sqrt{\ln n}/\delta^{1/4} \right\rceil +  \left\lfloor \frac{2016 n \ln
n \ln \ln n}{\dd^{1+\epsilon}} \right\rfloor + 1$. This completes the
proof of the first statement in Theorem \ref{thm:delta}.
In case $\delta^{1+\epsilon} > 2016 n \ln n \ln \ln n$, $m(H)
\leq \lceil n/\dd + 5\sqrt{n/\dd} \sqrt{\ln n}/\delta^{1/4}\rceil
+ 1 = \lceil  n/(\dd+1) + n/(\dd(\dd+1)) +5\sqrt{n/\dd} \sqrt{\ln
n}/\delta^{1/4} \rceil + 1$.  Since $\delta^{1+\epsilon} > 2016 n \ln
n \ln \ln n$, the value of $n/(\dd(\dd+1)) + 5\sqrt{n/\dd} \sqrt{\ln
n}/\delta^{1/4}$ is arbitrarily small
when $n$ is sufficiently large. Thus in this case,
$m(H) \leq \lceil  n/(\dd+1) \rceil +2$, as needed.

To see the first statement in Theorem \ref{t14} holds, notice that when
$\dd^{1+\epsilon} \geq  2016 n \ln n \ln \ln n$
then it is implied by the
second statement in  Theorem \ref{thm:delta}. Otherwise it follows 
from the first statement of this theorem and the 
fact that we may assume that $\dd \geq \Omega((n/\log n)^{1/4})$
by the results in Section 2, that $m(H)
\leq (n/\dd) (1+o(1)) = (n/(\dd+1))(1+o(1))$.  The second statement
in Theorem \ref{t14} holds since the condition $\delta^{1.24} \geq
n$ implies that for sufficiently large $n$, $\delta^{1.245} > 2016 n
\ln n \ln \ln n$ and the desired result follows from the second
statement in Theorem \ref{thm:delta}.
\hfill $\Box$

\section{Open problems}

The two conjectures \ref{c11} and \ref{c12} remain open, although
we have established some weaker asymptotic versions. It is possible
that the constant $2$ in both conjectures can even be replaced by 
$1$ provided the number of vertices in the graphs considered is 
large.  It may be
interesting to prove that the assertions of the two conjectures
hold if we replace the constant $2$ in each of them by some absolute
constant $C$. 
It will also be nice to prove that every
$d$-regular graph on $n$ vertices, where $d=o(n)$, contains a
spanning subgraph $H$ in which every degree between $0$ and $d$
appears $(1+o(1))\frac{n}{d+1}$ times,
even when $d$ is nearly linear in $n$. 
As is the case throughout 
the paper, the $o(1)$-term here tends to $0$ as $n$ tends
to infinity. Finally, Theorems \ref{t17} and 
\ref{t18} suggest the question of deciding whether or not 
there is an absolute constant $C$ so that
every graph $G$ (with a finite irregularity strength $s(G)$)
contains a spanning subgraph $H$ satisfying
$m(H) \leq s(G)+C$.
\vspace{0.1cm}

\noindent
{\bf Acknowledgment:}\,  We thank  D\"om\"ot\"or P\'alv\"olgyi 
for pointing out that the original version of our conjectures has been 
too strong.

\end{document}